\documentclass[11pt]{amsart}
\usepackage{array, amsmath, color}
\usepackage{amssymb, amsaddr, fullpage}
\usepackage{mathrsfs}
\usepackage{graphicx,subfigure}

\makeatletter
 \usepackage{amsthm}

\usepackage{fullpage}

\makeatother

\newtheorem{thm}{Theorem}[section]
\newtheorem{lem}[thm]{Lemma}

\newtheorem{definition}[thm]{Definition}

\title{Maximum average degree and relaxed coloring}

\author{Michael Kopreski and Gexin Yu}
\address{Department of Mathematics\\
 The College of William and Mary\\
 Williamsburg, VA, 23185, USA}

\thanks{The research was partially supported by the NSA grant  H98230-16-1-0316 and NSF EXTREEMS-QED grant DMS-1331921. The first author was supported an honor fellowship at the College of William and Mary. }

\email{gyu@wm.edu}

\begin{document}
\maketitle

\begin{abstract}
We say a graph is $(d, d, \ldots, d, 0, \ldots, 0)$-colorable with $a$ of $d$'s and $b$ of $0$'s if $V(G)$ may be partitioned into $b$ independent sets $O_1,O_2,\ldots,O_b$ and $a$ sets $D_1, D_2,\ldots, D_a$ whose induced graphs have maximum degree at most $d$.   The maximum average degree, $mad(G)$, of  a graph $G$ is the maximum average degree over all subgraphs of $G$.  In this note, for nonnegative integers $a, b$, we show that if $mad(G)< \frac{4}{3}a + b$, then $G$ is $(1_1, 1_2, \ldots, 1_a, 0_1, \ldots, 0_b)$-colorable.
\medskip


\noindent{\bf AMS Subject class:} 05C15
\end{abstract}

\section{Introduction}
We say a graph is $(d, d, \ldots, d, 0, \ldots, 0)$-colorable with $a$ of $d$'s and $b$ of $0$'s if $V(G)$ may be partitioned into $b$ independent sets $O_1, \ldots, O_b$ and $a$ sets $D_1, D_2,\ldots, D_a$ whose induced graphs have maximum degree at most $d$. For convenience, we also call the graph to be $(d_1, d_2, \ldots, d_a, 0_1, \ldots, 0_b)$-colorable.   The maximum average degree, $mad(G)$, of  a graph $G$ is the maximum average degree over all subgraphs of $G$.  This parameter is used to measure how sparse a graph is.

Borodin and Kostochka (2011, \cite{BK11}) showed that $mad(G)\le \frac{12}{5}$, then $G$ is $(1,0)$-colorable, and the upper bound $\frac{12}{5}$ is sharp; and for $d\ge 2$, they (2014,~\cite{BK14}) showed that if $mad(G)\le 3-\frac{1}{d+1}$, then $G$ is $(d,0)$-colorable, and again the upper bound $3-\frac{1}{d+1}$ is sharp.   Borodin, Kostochka, and Yansey~(2013, \cite{BKY13}) also gave the sharp result that if $mad(G)\le \frac{14}{5}$, then $G$ is $(1,1)$-colorable.

Havet and Sereni (2006,~\cite{HS06}) proved that if $mad(G)<a+\frac{ad}{a+d}$, then $G$ is $(d_1, d_2, \ldots, d_a)$-colorable.    For nonnegative integers $a, b, d$, Dorbec-Kaiser-Montassier-Raspaud (2014~\cite{DKMR14}) proved that a graph $G$ is $(d_1, d_2, \ldots, d_a, 0_1, \ldots, 0_b)$-colorable if $mad(G)<a + b + \frac{da(a+1)}{(a+d+1)(a+1) + ab}.$    Montassier and  Ochem (2015,~\cite{MO15}) gave a good survey on the results of this kind.

Clearly $(1_1, 1_2, \ldots, 1_a, 0_1, \ldots, 0_b)$-colorable graphs are also $(d_1, d_2, \ldots, d_a, 0_1, \ldots, 0_b)$-colorable for every $d\ge 1$.  When $d=1$, the above theorem shows that a graph $G$ is $(1_1, 1_2, \ldots, 1_a, 0_1, \ldots, 0_b)$-colorable if $mad(G)<a+b+\frac{a(a+1)}{(a+1)(a+2)+ab}<a+b+1$. In this note, we improve the upper bound by replacing the $1$ by $a/3$, which improves the previous result as long as $a>1$ or $b>0$. 

\begin{thm}\label{main}
Let $a, b$ be nonnegative integers  with $a\ge 1$. Then
\begin{center}
if $mad(G)< \frac{4}{3}a + b$, then $G$ is $(1_1, 1_2, \ldots, 1_a, 0_1, \ldots, 0_b)$-colorable.
\end{center}
\end{thm}

Note that a complete graph with $2a+b+1$ vertices has maximum average degree $2a+b$ and is not $(1_1, 1_2, \ldots, 1_a, 0_1, \ldots, 0_b)$-colorable. So the best possible improvement over our result $\frac{4}{3}a+b$ would be $2a+b$. It is an interesting question to find optimal upper bound on maximum average degrees.  

Our proof uses a non-traditional discharging method. Instead of distributing charges among local neighborhoods, we define a flow of charges among subsets of vertices.  The global discharging method allows us to prove a stronger result without lengthy discussion.

\section{The proof}
Let $G$ be a counterexample with fewest vertices.  Then for each vertex $x$ in $G$, $G-x$ has a desired coloring.   In a coloring of $G-x$, a vertex is \textit{saturated} if it is colored $0_i$ or colored $1_j$ with $1_j$-colored neighbor, and it is called $0$-saturated and $1$-saturated, respectively.

\begin{lem}\label{lem:min}
$\delta(G) \geq a + b$.
\end{lem}

\begin{proof}
By minimality of $G$, $G-v$ can be colored.  The coloring of $G-v$ cannot be extended to $v$.  Then all the $a+b$ colors must appear in $N(v)$. So $v$ has at least $a+b$ neighbors.
\end{proof}

For a vertex $v\in G$,  let $h(v)= d(v) - (a + b)$.

\begin{lem}\label{lem:satnbr}
Let $v\in V(G)$ be a vertex with $h(v)<a+b$. Then in a coloring of $G-v$,  $v$ has at least $a+b-h(v)$ saturated neighbors whose colors are unique in $N(v)$, and among them, at least $\max\{a - h(v),0\}$ are $1$-saturated.
\end{lem}

\begin{proof}
Let $S$ be the set of neighbors of $v$ whose colors appear only once in $N(v)$.  Note that each vertex in $S$ must be saturated, for otherwise, we may color $v$ with the color of a non-saturated vertex in $S$.

Suppose that $s=|S|<a+b-h(v)$. Since $v$ cannot be colored, all $a+b$ colors should appear in $N(v)$. So $a+b-s$ colors appear at least twice in $N(v)$. Therefore, $d(v)\ge 2(a+b-s)+s=a+b+(a+b-s)>a+b+h(v)=d(v)$, a contradiction.  This shows the first part of the lemma.

Now let $v$ such that $h(v)<a$, and suppose that $v$ has $p<a-h(v)$ $1$-saturated neighbors whose colors appear once in $N(v)$. Note that all $a+b$ colors must appear on $N(v)$, and each of the non-saturated neighbors must share the same color with another neighbor. Therefore, $d(v)\ge 2(a-p)+b+p=a+b+(a-p)>a+b+h(v)=d(v)$, a contradiction.
\end{proof}

\begin{definition}
Let $F_0$ be the set of vertices with degrees at least $2a+2b$.  For $k\ge 0$,  we define $H_k\subseteq V(G)-F_k$ to be a set of vertices with the following properties:

\begin{enumerate}
\item $v\in H_k$ has at least $\max\{a-h(v),0\}$ neighbors in $F_k$, and
\item $v\in H_k$ has at least $a + b - h(v)$ neighbors in the induced subgraph $G[H_k\cup F_k]$.
\end{enumerate}
Let $F_{k+1}=F_k\cup H_k$.
\end{definition}

We should note that $G$ may not have vertices of degree at least $2a+2b$, thus $F_0$ could be an empty set.  However, we shall show in the lemma below that $H_0$ can be chosen to be non-empty (that is, there exists vertices with $h(v)\le a$ and satisfying the second condition).

\begin{lem}\label{lem:flagged}
For some $k\ge 0$, $F_k=V(G)$.
\end{lem}

\begin{proof}
Suppose that $F_k\not=V(G)$ for every $k\ge 0$.   Consider the largest subset $F_k$ in $G$.

In all colorings of $G-v$ with $v\not\in F_k$,  the ones with minimized number of $1$-saturated vertices in $V(G)-F_k$ are called {\em minimum partial coloring} of $G$.  Not every vertex $v$ in $V(G)-F_k$ can make $G-v$ to have a minimum partial coloring of $G$, but since $G$ is finite, some vertices do.   Let $H$ be the set of vertices in $V(G)-F_k$ such that $G-v$ for $v\in H$ has a minimum partial coloring. Then for each $v\in H$,

\begin{enumerate}
\item $v$ has at least $\max\{a-h(v),0\}$ neighbors in $F_k$.

We assume that $a-h(v)>0$, and suppose that $v$ has fewer than $a-h(v)$ neighbors in $F_k$, which includes the case that $F_0=\emptyset$.  Consider a minimum partial coloring $c(G-v)$.  By Lemma~\ref{lem:satnbr}, $v$ has at least $a-h(v)$ $1$-saturated neighbors whose colors appear once in $N(v)$, and so one of them, say $u$, must be in $V(G)-F_k$.  Uncolor $u$ and color $v$ with the color of $u$, we obtain a coloring of $G-u$ with $u\in V(G)-F_k$.  In this coloring, $v$ is not saturated and no other vertices become saturated; so this coloring has fewer $1$-saturated  vertices than the one of $G-v$, a contradiction to the minimality of the coloring. Therefore, $v$ has at least $\max\{a-h(v), 0\}$ neighbors in $F_k$.

\item $v$ has at least $a+b-h(v)$ neighbors in the induced graph $G[H\cup F_k]$.

Since $v\in H$, $v\not\in F_0$, so $d(v)<2a+2b$, thus $a+b-h(v)>0$.    Consider a minimum coloring of $G-v$.  By Lemma~\ref{lem:satnbr}, $v$ has at least $a+b-h(v)$ saturated neighbors whose colors appear once in $N(v)$. We claim that all such neighbors are in $H\cup F_k$. For otherwise, let $w\not\in H\cup F_k$ be such a neighbor.  Uncolor $w$ and color $v$ with the color of $w$, we obtain a coloring of $G-w$ with $w\in V(G)-F_k$. Now, $v$ is not $1$-saturated and no other vertices become $1$-saturated.  So this coloring is also a minimum partial coloring with $w\in V(G)-F_k$.  By the definition of $H$, $w\in H$, a contradiction.
\end{enumerate}

Hence, by definition, $H\subseteq H_k$ (true even if $k=0$ and $F_0$ is empty), a contradiction.
\end{proof}

Now, we are ready to prove Theorem~\ref{main}.

For each vertex $v\in G$, let $\mu(v)=d(v)-(\frac{4}{3}a+b)=h(v)-\frac{1}{3}a$.  Since $mad(G)<\frac{4}{3}a+b$, $$\sum_{v\in V(G)} \mu(v)=\sum_{v\in V(G)} (d(v)-(\frac{4}{3}a+b))=2|E(G)|-(\frac{4}{3}a+b)|V(G)|<0.$$

We distribute the charges among vertices by the following rule:
\begin{enumerate}
\item[(R)] if $v\in F_k$ with $k=0$ or $v\in F_{k} - F_{k-1}$ for some $k > 0$, then $v$ gives $\frac{1}{3}$ to each $u \in N(v) - F_k$.
\end{enumerate}

\medskip

For vertex $v\in F_0$, $v$ gives at most $\frac{1}{3}$ to each of its neighbors. Note that $d(v)\ge 2a+2b$,  so the final charge of $v$ is $$\mu^*(v)\ge d(v)-(\frac{4}{3}a+b)-\frac{1}{3} d(v)\ge \frac{1}{3}(2(2a+2b)-4a-3b)>0.$$

Let $v\in F_k-F_{k-1}$ with $k\ge 1$. Then $v\in H_k$.  By the rule and definition of $H_k$, $v$ receives at least $\frac{1}{3}\max\{a-h(v),0\}$ from the neighbors in $F_{k-1}$ and gives out $\frac{1}{3}$ to each of its neighbors not in $F_k$.  By definition, $v$ has at most $d(v)-(a+b-h(v))=2h(v)$ neighbors not in $F_k$.    So the final charge of $v$ is
\begin{align*}
\mu^*(v)&\ge (h(v)-\frac{1}{3}a)+\frac{1}{3}\max\{a-h(v),0\}-\frac{1}{3}\cdot 2h(v)\\
&\ge \begin{cases}
	(h(v)-\frac{1}{3}a)+\frac{1}{3}(a-h(v))-\frac{2}{3}h(v)=0, \text{ if $a\ge h(v)$,}\\
	(h(v)-\frac{1}{3}a)+0-\frac{2}{3}h(v)=\frac{1}{3}(h(v)-a)>0, \text{ if $a<h(v)$.}
	\end{cases}			
\end{align*}

Therefore, every vertex in $G$ has a nonnegative final charge.  But $$\sum_{v\in V(G)}\mu^*(v)=\sum_{v
\in V(G)} \mu(v)<0,$$ we reach a contradiction.  This contradiction shows the truth of Theorem~\ref{main}.

\bigskip

{\bf Acknowledgement:} 
 We thank the referees for their helpful comments.

\end{document}